\documentclass[review,12pt]{elsarticle}

\usepackage{amssymb}
\usepackage{amsmath}
\usepackage{subfigure}
\usepackage{color}
\usepackage{mathrsfs} 
\usepackage[colorlinks=true]{hyperref}
\usepackage{amsthm}

\usepackage{bm}
\usepackage{cleveref}
\usepackage{multirow}

\usepackage{algorithm}
\usepackage{algorithmicx}
\usepackage{algpseudocode}

\usepackage{booktabs}
\usepackage[normalem]{ulem}

 \newtheorem{thm}{Theorem}[section]
 
 \newtheorem{lem}[thm]{Lemma}
 \newtheorem{prop}[thm]{Proposition}

 \theoremstyle{remark}
 \newtheorem{rem}[thm]{Remark}

\def\pt{\partial}

\def\bff{\mathbf{f}}

\def\im{\mathrm{i}}

\def\half{\frac{1}{2}}

\def\R{\mathbb{R}}

\def\dtdtau{\frac{\Delta t}{\tau}}

\def\sign{\mathrm{sign}}
\def\erf{\mathrm{erf}}

\def\bfU{\mathbf{U}}
\def\bfu{\mathbf{u}}
\def\bfA{\mathbf{A}}
\def\bfb{\mathbf{b}}
\def\bfl{\mathbf{l}}

\def\bfL{\mathbf{L}}

\def\bffR{{\mathbf{\mathcal{R}}}}
\def\bfx{\mathbf{x}}
\def\bfy{\mathbf{y}}

\begin{document}

\begin{frontmatter}



\title{Stability of the first-order unified gas-kinetic scheme based on a linear kinetic model}

\author[HKUSTm]{Tuowei Chen\corref{cor1}}

\ead{tuowei_chen@163.com}

\author[HKUSTm,HKUSTmae,HKUSTsz]{Kun Xu}

\ead{makxu@ust.hk}

\cortext[cor1]{Corresponding Author}

\address[HKUSTm]{Department of Mathematics, Hong Kong University of Science and Technology, HongKong, China }

\address[HKUSTmae]{Department of Mechanical and Aerospace Engineering, Hong Kong University of Science and Technology, Hong Kong, China}

\address[HKUSTsz]{HKUST Shenzhen Research Institute, Shenzhen 518057, China}

\begin{abstract}
The unified gas-kinetic scheme (UGKS) is becoming increasingly popular for multiscale simulations in all flow regimes. This paper provides the first analytical study on the stability of the UGKS applied to a linear kinetic model, which is able to reproduce the one-dimensional linear scalar advection-diffusion equation via the Chapman-Enskog expansion method. Adopting periodic boundary conditions and neglecting the error from numerical integration, this paper rigorously proves the weighted $L^2$-stability of the first-order UGKS under the Courant–Friedrichs–Lewy (CFL) conditions. It is shown that the time step of the method is not constrained by being less than the particle collision time, nor is it limited by parabolic type CFL conditions typically applied in solving diffusion equations. The novelty of the proof lies in that based on the ratio of the time step to the particle collision time, the update of distribution functions is viewed as a convex combinations of sub-methods related to various physics processes, such as the particle free transport and collisions. The weighted $L^2$-stability of the sub-methods is obtained by considering them as discretizations to corresponding linear hyperbolic systems and utilizing the associated Riemann invariants. Finally, the strong stability preserving property of the UGKS leads to the desired weighted $L^2$-stability.
\end{abstract}



\begin{keyword}
BGK model, unified gas-kinetic scheme, strong stability



\end{keyword}

\end{frontmatter}

\section{Introduction}
Gas dynamics can be modeled with a variable scale in different flow regimes. For the continuum flow, the Navier-Stokes (NS) equations are the well-known macroscopic governing equations for fluid dynamics, while in the rarefied regime, the Boltzmann equation is the fundamental equation for describing nonequilibrium flows. To solve the Boltzmann equation, the direct simulation Monte Carlo (DSMC) \cite{Bird1994} is a leading stochastic method for high-speed rarefied flow. In the framework of deterministic approximation, the most popular class of methods is the discrete velocity method (DVM) \cite{Mieussens2000}, which directly discretizes both the spatial domain and the particle velocity space. Both DSMC and conventional DVM are based on the operator splitting treatment for the particle transport and collision. This requires the cell size and time step to be less than the particle mean free path and collision time in the explicit numerical evolution process, making these methods prohibitive in the continuum flow application. For continuum flow computations, the gas-kinetic scheme (GKS) \cite{Xu2001} is a robust and accurate hydrodynamic flow solver for the NS solutions. The GKS uses the integral solution of the Bhatnagar-Gross-Krook (BGK) kinetic model \cite{BGK1954} along cell interfaces to obtain numerical fluxes. Since the GKS updates the macroscopic flow variables only and uses the Chapman–Enskog expansion \cite{CE1990} for the construction of a gas distribution function, it is only valid for the continuum flow. 

Numerical methods for solving the NS equations and the Boltzmann equation are typically confined to simulating gas dynamics at hydrodynamic and kinetic scales, respectively. These methods can be reliably applied within their specific scales due to the clear scale separation. However, in real science and engineering applications, this separation is often not so distinct. For instance, around a hypersonic flying vehicle, different flow physics may emerge at different regions, such as the highly non-equilibrium shock layer, low density trailing edge, and the wake turbulence, corresponding to different regimes, where the local Knudsen number can vary significantly in several orders of magnitude \cite{XL2017}. This complexity highlights the necessity for developing numerical methods that can effectively address the entire flow regime.

In recent years, with the discretized particle velocity space and based on the BGK model, a unified gas-kinetic scheme (UGKS) has been proposed for multiscale simulations in all flow regimes \cite{XH2010,HXY2012,Xu2015}. Compared with the GKS, the UGKS updates both macroscopic variables and the gas distribution function, and computes the numerical flux without using the Chapman-Enskog expansion, removing the hydrodynamic scale limitation. In comparison with the DVM, the UGKS takes into account the effect of particle collisions around cell interfaces, eliminating the restriction on the time step and cell size being less than the collision time and mean free path. Recently, the UGKS has been successfully used in simulations of radiative transfer \cite{SJX2017}, plasma \cite{LX2017} and particle flows \cite{LWX2019}. The asymptotic preserving property of the UGKS for diffusion limit of a linear transport model was analyzed in \cite{Mieussens2013}. The unified preserving property of the discrete UGKS was studied in \cite{GLX2023}. To our knowledge, no analytical investigations on the stability of this class of multiscale kinetic methods are available so far. It is worth to mention that the $L^2$-stability of the GKS applied to the linear advection-diffusion equation was established in \cite{TX2006}. 

The purpose of this paper is to present some analytical results about the numerical stability of the UGKS in a simplified setting. We are not concerned with the competitiveness of the UGKS for the full compressible gas dynamics, however, our results may provide some explanations for the robustness of the method. We apply the explicit UGKS method to a linear kinetic model. This model is of BGK type and is able to approximate the one-dimensional (1-D) linear scalar advection-diffusion equation via the Chapman-Enskog expansion. For simplicity, we consider only the first-order accurate method and adopt periodic boundary conditions. The resulting first-order UGKS allows to investigate the structure of the UGKS in a most explicit setting which will provide a thorough understanding of the multiscale kinetic mechanisms in the method. This paper provides a novel analytical framework to study the $L^2$-stability of multiscale kinetic methods. The results serve as an initial step toward a comprehensive analysis of the full gas-dynamic case.

The main contribution of this paper is highlighted as follows. First, we demonstrate that the first-order UGKS based on a linear kinetic model has a constraint-preserving property, assuming that the error from numerical integration is negligible. Second, based on the ratio of the time step to the particle collision time, we view the the first-order UGKS as a convex combination of sub-methods that address various physical processes, including particle-free transport across cell interfaces, particle collisions around cell interfaces, and particle collisions within each cell. This approach enables us to analyze the strong stability of those physics-process-related sub-methods respectively, leveraging the strong stability preserving property of the method. Third, to derive the strong stability of the sub-methods, we consider them as discretizations to linear hyperbolic systems and utilize the associated Riemann invariants to construct a weighted $L^2$ convex functional. As a result, the weighted $L^2$-stability of the first-order UGKS is rigorously proved under the Courant–Friedrichs–Lewy (CFL) conditions. The time step of the method neither suffers limitations of being less than the collision time, nor is it restricted by parabolic type CFL conditions typically used in solving diffusion equations.

The rest of this paper is organized as follows. In \Cref{sec:model}, a linear kinetic model of BGK type is introduced. In \Cref{sec:scheme}, based on the linear kinetic model, a first-order UGKS is proposed. In \Cref{sec:analysis}, the strong stability of the proposed scheme is rigorously proved. Finally, conclusions are presented in \Cref{sec:conclusion}.

\section{A linear kinetic model}\label{sec:model}
In this section, a linear kinetic model is introduced to reproduce the 1-D scalar, linear advection-diffusion equation via the Chapman-Enskog expansion method.

\subsection{Hydrodynamic equation and kinetic model equation}
The governing equations for gas dynamics are obtained based on different scale modeling. 

In the hydrodynamic scale, the 1-D scalar, linear advection-diffusion equation reads  
\begin{equation}\label{eq:lad}
	\pt_t u + a\pt_x u= \nu \pt_x^2 u,
\end{equation} 
where $u\in\R$ is the conservative variable, $a>0$ is a constant standing for the advection velocity, and $\nu\in\R$ is a positive viscosity or diffusion coefficient. This equation can be regarded as a toy model for viscous fluid flows with zero pressure.

In the kinetic mean free path scale, the evolution of the gas distribution function can be described by the BGK equation,
\begin{equation}\label{eq:bgk}
	\pt_t f+ c\pt_x f = \frac{1}{\tau}(g-f),
\end{equation}
where $f(c,x,t)$ is the gas distribution function, $c$ is microscopic particle velocity, $g$ is the equilibrium state, and $\tau$ is a positive constant standing for the mean collision time. Different choices for the equilibrium state $g$ lead to models for various physical equations, such as gas dynamics and shallow-water equations. For the full gas dynamic equation the equilibrium is given by a Maxwell distribution.

The macroscopic and microscopic quantities are related through
\begin{equation}\label{conservative-moment}
	u(x,t)=\int_{\R}f(c,x,t)dc,
\end{equation}
which links the kinetic description to the macroscopic variable $u$.

\subsection{Kinetic model for the linear advection-diffusion equation}

To reproduce the 1-D scalar, linear advection-diffusion equation \eqref{eq:lad} in the spirit of \cite{KXMJ1997,TX2006,XL2017}, the equilibrium state $g$ in \eqref{eq:bgk} is set to be
\begin{equation}\label{eq:g}
	g(c,x,t)= u(x,t)\frac{1}{\sqrt{\theta\pi}}
	e^{-\frac{(c-a)^2}{\theta}}.
\end{equation} 
In comparison with the Maxwellian distribution, $u$ corresponds to the density, $a$ to the mean field velocity, and $\theta$ to the temperature. With the equilibrium state \eqref{eq:g}, the BGK model \eqref{eq:bgk} becomes a linear kinetic model.

By the construction \eqref{eq:g}, the equilibrium distribution $g$ satisfies 
\begin{equation}
	\int_{\R}g(c,x,t)dc=u(x,t).
\end{equation} 
Therefore, integrating with respect to the particle velocity on both sides of the equation \eqref{eq:bgk} gives the macroscopic equation
\begin{equation}
	\pt_t u + \pt_x F = 0,
\end{equation}
with the macroscopic flux
\begin{equation}\label{macroscopic-flux}
	F(x,t)=\int_{\R}cf(c,x,t) dc.
\end{equation}

For smooth solutions, the Chapman-Enskog expansion method \cite{CE1990} considers an asymptotic expansion
\begin{equation}
	f^{(N)}(c,x,t)
	=\sum_{n=0}^{N}\tau^{n} f_{n}(c,x,t)
\end{equation}
of the distribution function in terms of the relaxation time $\tau$ up to an order $N$. The parameter $\tau$ is assumed to be small in an appropriate dimensionless scaling. The first expansion becomes
\begin{equation}\label{CEfirst}
	f^{(1)}(c,x,t)=(u-\tau(c-a)\pt_x u)
	\frac{1}{\sqrt{\theta\pi}}e^{-\frac{(c-a)^2}{\theta}}.
\end{equation}
The corresponding macroscopic flux is 
\begin{equation}
	F^{(1)}=au-\frac{\theta\tau}{2} \pt_x u,
\end{equation}
which leads to the advection-diffusion equation
\begin{equation}\notag
	\pt_t u + a\pt_x u = \frac{\theta\tau}{2} \pt_x^2 u.
\end{equation}  
Comparing the coefficients in the above equation with the equation \eqref{eq:lad} gives the relation 
\begin{equation}
	\nu= \frac{\theta\tau}{2}.
\end{equation}
The expansion \eqref{CEfirst} converges when $\tau$ is small, and the equation \eqref{eq:lad} may not be consistent with the equation \eqref{eq:bgk} when $\tau$ gets large.

\section{The first-order unified gas-kinetic scheme}\label{sec:scheme}
This section is devoted to the first-order UGKS based on the linear kinetic model \eqref{eq:bgk} with an equilibrium state \eqref{eq:g}. 

\subsection{Finite volume framework}
The UGKS is a direct physical modeling method for the time evolution of both the gas distribution function and macroscopic flow variables in a discretized space. In one-dimensional space, the spatial space $\Omega$ is divided into uniform control volumes, i.e., $\Omega_{i}=[x_{i-\half},x_{i+\half}]$ with a cell center $x_i$ and cell size $\Delta x=x_{i+\half}-x_{i-\half}$ for $i=1,\dots,I$. The temporal discretization is denoted by $t^n$ for the $n$-th time step. The particle velocity space is discretized by $2K + 1$ Cartesian mesh points with a uniform velocity spacing $\Delta c$, and the center of the $k$-th particle velocity interval is $c_k = a+k\Delta c$ for $k=-K,\dots,K$.

The time evolution of a gas distribution function in the computational space is due to the particle transport through cell interfaces and the particle collisions inside each cell. For a gas distribution function $f$ in a control volume around space $x_{i}$, time $t^{n}$, and particle velocity $c_k$, the direct modeling method gives
\begin{equation}\label{eq:physics-distribution}
	\begin{aligned}
		f^{n+1}_{k,i}=&f^{n}_{k,i}
		-\frac{1}{\Delta x}\int_{t^n}^{t^{n+1}}
		\big(c_{k}f_{k,i+\half}(t)-c_{k}f_{k,i-\half}(t)\big)dt\\
		&+\int_{t^n}^{t^{n+1}}\int_{\Omega_{i}} Q(f)dxdt,
	\end{aligned}
\end{equation}
where $f^{n}_{k,i}$ and $f^{n+1}_{k,i}$ are the cell average of $f(x,t,c_k)$ over the cell $\Omega_{i}$ at time $t=t^{n}$ and $t^{n+1}$, respectively. Moreover, $f_{k,i+\half}(t)$ is a time-dependent distribution function on the cell interface $x=x_{i+\half}$ at particle velocity $c_k$, and $Q(f)$ is the time-dependent particle collision term inside each cell, which redistributes the particle in the velocity space. For the BGK collision model \eqref{eq:bgk}, $Q(f)=(g-f)/\tau$.

The time evolution of macroscopic conservative variables is due to the conservation of conservative variables during particle collisions. The update of the conservative moment given by \eqref{conservative-moment} becomes
\begin{equation}\label{eq:physics-conservation}
	u^{n+1}_{i}=u^n_{i}
	-\frac{1}{\Delta x}
	\int_{t^n}^{t^{n+1}}\big(F_{i+\half}(t)-F_{i-\half}(t)\big)dt,
\end{equation}
where $u^{n}_{i}$ and $u^{n+1}_{i}$ are the cell average of the conservative variable $u$ over the cell $\Omega_{i}$ at time $t=t^{n}$ and $t^{n+1}$, respectively. In addition, $F_{i+\half}(t)$ is the time-dependent macroscopic flux determined by the first moment of the distribution function at the cell interface $x=x_{i+\half}$. Note that both governing equations \eqref{eq:physics-distribution} and \eqref{eq:physics-conservation} are the basic physical laws in the mesh size and time step scales, and there is no inaccuracy introduced yet.

In order to evolve the above two governing equations in the finite volume framework, it remains to determine the numerical fluxes at cell interfaces and the numerical source for particle collision terms by using a direct modeling method. The modeling of the numerical flux and source are based on the linear kinetic model \eqref{eq:bgk} with an equilibrium state \eqref{eq:g}, and the details are presented in the following subsections.

\subsection{Numerical flux}
The key ingredient of UGKS is the construction of the time-dependent discrete distribution functions $f_{k,i+\half}(t)$ at the cell interface. It is obtained by using the integral representation of the solution of the BGK equation \eqref{eq:bgk} that
\begin{equation} \label{integral-solution}
	\begin{aligned}
		f_{k,i+\half}(t)=f(x_{i+\half},t,c_k)
		=&\frac{1}{\tau}\int_{t^n}^{t}g(x-c_k(t-s),s,c_k)e^{-(t-s)/\tau}ds\\
		&+e^{-(t-t^{n})/\tau}f(x_{i+\half}-c_k(t-t^{n}),t^n,c_k).
	\end{aligned}
\end{equation}
Here, since we consider only first-order accurate methods in the present paper, the initial distribution function around the cell interface are constructed as
\begin{equation}
	f(x,t^n,c_k)=
	\begin{cases}
		f^n_{k,i},\quad\quad  &x<x_{i+\half},\\
		f^n_{k,i+1},\quad\quad &x\geq x_{i+\half}.
	\end{cases}	
\end{equation}
For a first-order accurate method, the equilibrium distribution function $g$ around $(x_{i+\half},t^n)$ is assumed to be constant in space and time, yielding 
\begin{equation}
	g(x,t,c_k)=g^n_{k,i+\half},
\end{equation}
where $g^n_{k,i+\half}$ is the equilibrium state on the cell interface $x=x_{i+\half}$, at particle velocity $c_k$. 

In the above equation, the construction of $g^n_{k,i+\half}$ depends on the modeling of particle collisions around the cell interface. For the linear kinetic model \eqref{eq:bgk} with an equilibrium state \eqref{eq:g}, $g^n_{k,i+\half}$ takes the form
\begin{equation}\label{g-form}
	g^n_{k,i+\half}=u^{g,n}_{i+\half}\frac{1}{\sqrt{\theta\pi}}e^{-\frac{(c_k-a)^2}{\theta}}.
\end{equation}
The value of $u^{g,n}_{i+\half}$ is determined using a particle velocity-weighted method,
\begin{equation}\label{ug-interface}
	u^{g,n}_{i+\half}=\sum_{k}\Delta c\frac{c_k}{\sqrt{a^2+\frac{1}{2}\theta}} f^{g,n}_{k,i+\half},
\end{equation}
where
\begin{equation}
	f^{g,n}_{k,i+\half} = 
	\frac{1}{2}(f^n_{k,i}+f^n_{k,i+1})
	-\erf(\frac{a}{\sqrt{\theta}})(f^n_{k,i+1}-f^n_{k,i}).
\end{equation}
In the above equations, $\erf(\cdot)$ is the error function given by
\begin{equation}\notag
	\erf(x)=\int_{0}^{x}e^{-t^2}dt.
\end{equation}

Combining the formulae \eqref{integral-solution}--\eqref{ug-interface} leads to the expression of numerical fluxes for updating the discretized distribution function at particle velocity $c_k$,
\begin{equation}\label{numerical-flux:f}
	\begin{aligned}
		f^{n,*}_{k,i+\half}=&\frac{1}{\Delta t}\int_{t^n}^{t^{n+1}}f_{k,i+\half}(t)dt\\
		=&\big(1-\frac{\tau}{\Delta t}(1-e^{-\frac{\Delta t}{\tau}})\big)
		g^n_{k,i+\half}+\frac{\tau}{\Delta t}(1-e^{-\frac{\Delta t}{\tau}})f^n_{k,i+\half},
	\end{aligned}
\end{equation}
where $\Delta t=t^{n+1}-t^{n}$, and 
\begin{equation}
	f^n_{k,i+\half}=\frac{1}{2}(f^n_{k,i}+f^n_{k,i+1})
	-\frac{1}{2}\sign(c_k)(f^n_{k,i+1}-f^n_{k,i}).
\end{equation}

For updating the conservative variable $u^{n+1}_{i}$ over the cell $\Omega_{i}$, the macroscopic numerical flux at the cell interface $x=x_{i+\half}$ is given by
\begin{equation}\label{numerical-flux:F}
	\begin{aligned}
		F^{n,*}_{i+\half}=\frac{1}{\Delta t}\int_{t^n}^{t^{n+1}}F_{i+\half}(t)dt
		=&\frac{1}{\Delta t}\int_{t^n}^{t^{n+1}}\big(\sum_{k}c_kf_{k,i+\half}(t)\Delta c\big)dt\\
		=&\sum_{k}c_kf^{n,*}_{k,i+\half}\Delta c.
	\end{aligned}
\end{equation}

\begin{rem}
	The construction of $u^{g,n}_{i+\half}$ in \eqref{ug-interface} is to ensure the $l^2$ stability for the update of the macroscopic momentum, i.e., the first moment $\sum_{k}\Delta c c_k f^{n}_{k,i}$, in the hydrodynamic limit $\tau\rightarrow0$. It can be verified that in the hydrodynamic limit $\tau\rightarrow0$, the update of the macroscopic momentum $\sum_{k}\Delta cc_k f^{n}_{k,i}$ is reduced to
	\begin{equation}
		F^{n+1}_{i}=F^{n}_{i}
		-\tilde{a}\frac{\Delta t}{\Delta x}
		\bigg(F^{g,n}_{i+\half}-F^{g,n}_{i-\half}\bigg),
	\end{equation}
	where 
	\begin{equation}\label{def:F-cell-interface}
		\begin{aligned}
			&F^{n+1}_{i}=\sum_{k}\Delta c c_kf_{k,i}^{n+1},\quad 
			F^n_{i}=\sum_{k}\Delta c c_kf_{k,i}^{n},\\
			&F^{g,n}_{i+\half}=\frac{1}{2}(F^n_{i}+F^n_{i+1})
			-\erf(\frac{a}{\sqrt{\theta}})(F^n_{i+1}-F^n_{i}),\\
			&\tilde{a}=\frac{1}{\sqrt{a^2+\frac{1}{2}\theta}}\sum_{k} \Delta c {c_k}^2\frac{1}
			{\sqrt{\theta\pi}}e^{-\frac{(c_k-a)^2}{\theta}}.
		\end{aligned}
	\end{equation}
	In the above equations, $F^{g,n}_{i+\half}$ is of the GKS type cell interface flux \cite{Xu2001,TX2006} for the first moment,  which is recovered by the construction of $u^{g,n}_{i+\half}$ in \eqref{ug-interface}. The $l^2$ stability of the finite volume scheme with the GKS type numerical flux was proved in \cite{TX2006}, which will play a key role in obtaining the strong stability of the current scheme in all flow regimes, see \Cref{subsect-interface-g} below. 
	
\end{rem}

\subsection{Numerical source}
It remains to approximate the integral of the collision term in \eqref{eq:physics-distribution}. Based on the linear kinetic model \eqref{eq:bgk} with an equilibrium \eqref{eq:g}, the source term $Q(f)=(g-f)/\tau$. In the near continuum flow regime, the time step size $\Delta t$ can be much larger than the particle mean collision time $\tau$. In order to overcome the difficulty due to the stiff collision effect in the near continuum flow regime, the explicit UGKS will update the macroscopic variables first. The update of the conservative variable $u$ inside the control volume $\Omega_{i}$ is 
\begin{equation}\label{FV:u}
	u^{n+1}_{i}=u^n_{i}
	-\frac{\Delta t}{\Delta x}\big(F^{n,*}_{i+\half}-F^{n,*}_{i-\half}\big),
\end{equation} 
where the macroscopic numerical flux $F^{n,*}_{i+\half}$ is given in \eqref{numerical-flux:F}.

Based on the above updated conservative variable, the discrete equilibrium distribution function inside each cell at the next time level is predicted as
\begin{equation}\label{numerical-source:g}
	g^{n+1}_{k,i}=u^{n+1}_{i}\frac{1}{\sqrt{\theta\pi}}e^{-\frac{(c_k-a)^2}{\theta}}.
\end{equation}

Then, to achieve first-order accuracy and overcome the stiffness, the backward Euler method is applied to compute the numerical source in \eqref{eq:physics-distribution}, leading to 
\begin{equation}\label{FV:f-ori}
	\begin{aligned}
		f^{n+1}_{k,i}=f^{n}_{k,i}
		-\frac{c_{k}\Delta t}{\Delta x}
		\big( f^{n,*}_{k,i+\half}- f^{n,*}_{k,i-\half}\big)
		+\frac{\Delta t}{\tau}\big(g^{n+1}_{k,i}-f^{n+1}_{k,i}\big),
	\end{aligned}
\end{equation}
where the numerical flux for the distribution function at particle velocity $c_k$ is given in \eqref{numerical-flux:f}. The equation \eqref{FV:f-ori} indeed gives an explicit finite volume scheme for the discrete distribution function at particle velocity $c_k$, 
\begin{equation}\label{FV:f}
	f^{n+1}_{k,i}=\frac{1}{1+\frac{\Delta t}{\tau}}\bigg[f^n_{k,i}
	-\frac{c_{k}\Delta t}{\Delta x}
	\big( f^{n,*}_{k,i+\half}-f^{n,*}_{k,i-\half}\big)
	\bigg]+\frac{\frac{\Delta t}{\tau}}{1+\frac{\Delta t}{\tau}}g^{n+1}_{k,i}.
\end{equation}

\subsection{Summary of the numerical scheme}
Finally, the first-order explicit UGKS developed so far is summarized as \Cref{alg:scheme}.

\begin{algorithm}[!htbp]
	\caption{First-order UGKS} \label{alg:scheme}
	\begin{algorithmic}[1]
		\State At each cell interface, compute the numerical flux for discrete distribution functions and the conservative variable by \eqref{numerical-flux:f} and \eqref{numerical-flux:F}.
    	\State In each cell $\Omega_i$, update the conservative variable $u^{n+1}_{i}$ by \eqref{FV:u}. Compute the equilibrium distribution function $g^{n+1}_{k,i}$ by \eqref{numerical-source:g} for all $k$.
		\State Update the discrete distribution functions $f^{n+1}_{k,i}$ by \eqref{FV:f} for all $k$.
	\end{algorithmic}
\end{algorithm}

\section{Stability analysis} \label{sec:analysis}
This section is devoted to the weighted $L^2$ stability analysis of the current first-order UGKS.

\subsection{Main theorems}
This subsection is devoted to the main theorems of the present paper.

For simplicity, in the present paper the boundary conditions are assumed to be periodic and they are implemented using ghost cells,
\begin{equation}\label{boundary-periodic}
	u^{n}_{0}=u^{n}_{I},\quad u^{n}_{I+1}=u^{n}_{1},\quad
	f^{n}_{k,0}=f^{n}_{k,I},\quad f^{n}_{k,I+1}=f^{n}_{k,1}.
\end{equation}
The initial data are assumed to satisfy the conservation constraint
\begin{equation}\label{assumption-ini-compatibility}
	u^{0}_{i}=\sum_{k}f^{0}_{k,i}\Delta c.
\end{equation}
Here and in the sequel, the following notations are adopted,
\begin{equation}\notag
	W(\frac{\Delta t}{\tau})=\frac{\tau}{\Delta t}(1-e^{-\frac{\Delta t}{\tau}}),\quad
	\omega_k=\frac{1}{\sqrt{\theta\pi}}e^{-\frac{(c_k-a)^2}{\theta}}.
\end{equation}
Note that in practical computations, the particle velocity space is typically well discretized with a large $K$ to ensure that the numerical integration achieves the desired accuracy. Therefore, in this paper, we neglect the error associated with numerical integration and assume that it is accurate in the following sense,
\begin{equation}\label{assumption-num-integration}
	\sum_{k}\Delta c\omega_k=1,\quad \sum_{k}\Delta cc_k\omega_k=a,\quad
	\sum_{k} \Delta c {c_k}^2\omega_k=a^2+\frac{1}{2}\theta.
\end{equation}

Under the condition that the numerical integration is accurate, it is observed the first-order UGKS has the constraint-preserving property.
\begin{prop}
	Suppose that the condition \eqref{assumption-num-integration} holds. Then, the first-order UGKS has the constraint-preserving property, i.e.,
	\begin{equation}\label{constraint-preserving}
		u^{n+1}_{i}-\sum_{k} \Delta c f^{n+1}_{k,i}=0,  \quad\quad	i=1,\dots,I, \, n\geq 0,
	\end{equation}
	provided that
	\begin{equation}\label{conservative-n}
		u^{n}_{i}-\sum_{k}\Delta c f^{n}_{k,i}=0.
	\end{equation}	
\end{prop}
\begin{proof}
	Compute $u^{n+1}_{i}-\sum_{k} \Delta c f^{n+1}_{k,i}$ by using the equations \eqref{FV:u}, \eqref{numerical-flux:F} and \eqref{FV:f}, leading to
	\begin{equation}\notag
		\begin{aligned}
			u^{n+1}_{i}-\sum_{k} \Delta c f^{n+1}_{k,i}
			=&u^{n}_{i}-\sum_{k} \Delta c f^{n}_{k,i}-\frac{\Delta t}{\tau}\sum_{k}\Delta c
			\big(u^{n+1}_{i}\omega_k- f^{n+1}_{k,i}\big)\\
			=&u^{n}_{i}-\sum_{k} \Delta c f^{n}_{k,i}-\frac{\Delta t}{\tau}
			\big(u^{n+1}_{i}-\sum_{k} \Delta c f^{n+1}_{k,i}\big),
		\end{aligned}
	\end{equation}
	where the condition \eqref{assumption-num-integration} has been used. This implies 
	\begin{equation}\notag
		u^{n+1}_{i}-\sum_{k} \Delta c f^{n+1}_{k,i}
		=\frac{1}{1+\dtdtau}\big(u^{n}_{i}-\sum_{k} \Delta c f^{n}_{k,i}\big)=0,
	\end{equation}
	provided that \eqref{conservative-n} holds.
	
	The proof is completed.
\end{proof}

Adopt the notations
\begin{equation}\notag
	\bfU^{n}=\big(\frac{\bff^{n}_{-K}}{\sqrt{\omega_{-K}}},\dots,
	\frac{\bff^{n}_{K}}{\sqrt{\omega_{K}}}\big),\quad
	\bff^{n}_{k}=(f^{n}_{k,1},\dots,f^{n}_{k,I})^\top,\quad k=-K,\dots,K.
\end{equation}
Define the norms
\begin{equation}\label{def:norm}
	\|\bfU^n\|_{L^2}=\bigg(\sum_{k=-K}^{K}\Delta c
	\big\|\frac{\bff^{n}_{k}}{\sqrt{\omega_k}}\big\|^2_{l^2}\bigg)^{\half},\quad
	\|\bff^{n}_{k}\|_{l^2}=\big(\sum_{i=1}^{I}\Delta x
	|f^{n}_{k,i}|^2\big)^{\half}.
\end{equation}
The above $L^2$-norm of $\bfU^n$ is indeed a weighted $L^2$-norm of the discrete distribution functions. Then, the main theorem for the weighted $L^2$ stability of the first-order UGKS is stated as follows.
\begin{thm}
	Suppose that the conditions \eqref{assumption-num-integration} and \eqref{assumption-ini-compatibility} hold. Suppose that the periodic boundary conditions \eqref{boundary-periodic} holds.
	Then, if the time step $\Delta t$ satisfies the CFL condition
	\begin{equation}\label{CFL}
		\max\bigg\{\max_{-2K\leq k\leq 2K}|c_k|,
		\sqrt{a^2+\frac{1}{2}\theta}\big/\erf\big(\frac{a}{\sqrt{\theta}}\big) \bigg \}
		\frac{\Delta t}{\Delta x}\leq 1,
	\end{equation}	
	then the numerical solution of the first-order UGKS satisfies
	\begin{equation}\label{ineq-main-f}
		\|\bfU^{n+1}\|_{L^2}\leq \|\bfU^n\|_{L^2}.
	\end{equation}
	In addition, denote $\bfu^{n}=(u^n_{1},\dots,u^n_{I})$. The conservative variable $\bfu$ satisfies
	\begin{equation}\label{ineq-main-u}
		\|\bfu^{n+1}\|_{l^2}
		\leq \|\bfU^{n+1}\|_{L^2}
		\leq \|\bfU^n\|_{L^2}.
	\end{equation}
\end{thm}
\begin{proof}
	The inequality \eqref{ineq-main-f} follows from \Cref{lem:convex}, \Cref{lem:f}, \Cref{lem:g} and \Cref{lem:s} below.	
	
	Based on the conditions \eqref{assumption-num-integration} and \eqref{assumption-ini-compatibility}, the constraint-preserving property \eqref{constraint-preserving} leads to 
	\begin{equation}\label{ineq-temp-u}
		\begin{aligned}
			\sum_{i=1}^{I}|u^{n+1}_{i}|^2
			=&\sum_{i=1}^{I}\big|\sum_{k=-K}^{K}
			\Delta cf^{n+1}_{k,i}\big|^2
			=\sum_{i=1}^{I}\bigg|\sum_{k=-K}^{K}\Delta c\sqrt{\omega_k} \frac{f^{n+1}_{k,i}}{\sqrt{\omega_k}}\bigg|^2\\
			\leq& \sum_{i=1}^{I}\bigg[
			\big(\sum_{k=-K}^{K}\Delta c\omega_k\big)^{\half}
			\big(\Delta c
			\big\|\frac{\bff^{n+1}_{k}}{\sqrt{\omega_k}}\big\|^2_{l^2}\big)^{\half}
			\bigg]^2\\
			=& \|\bfU^{n+1}\|^2_{L^2}.
		\end{aligned}
	\end{equation}
	In the second line of the above inequality, the Cauchy-Schwartz inequality has been used. Therefore, \eqref{ineq-main-u} follows from \eqref{ineq-temp-u} and \eqref{ineq-main-f}.
	
	The proof is completed.
\end{proof}

\subsection{Convex combination of physics-process-related sub-methods}


A key ingredient of the numerical stability analysis is to view the current scheme \eqref{FV:f} as a convex combination of physics-process-related sub-methods. The updates of discrete distribution functions \eqref{FV:f} and the conservative variable \eqref{FV:u} are rewritten in the form
\begin{equation}\label{FV:convex}
	\begin{aligned}
		&f^{n+1}_{k,i}=W(\dtdtau)f^{f,n+1}_{k,i}
		+\frac{1}{1+\dtdtau}(1-W(\dtdtau))f^{g,n+1}_{k,i}
		+\frac{\dtdtau}{1+\dtdtau}f^{s,n+1}_{k,i},\\
		&u^{n+1}_{i}=W(\dtdtau)u^{f,n+1}_{i}
		+(1-W(\dtdtau))u^{g,n+1}_{i}, 
	\end{aligned}
\end{equation} 
where
\begin{equation}\label{FV:convex-f}
	\begin{aligned}
		&f^{f,n+1}_{k,i}=f^{n}_{k,i}
		-\frac{c_k\Delta t}{\Delta x}(f^{n}_{k,i+\half}-f^{n}_{k,i-\half}),\\
		&f^{g,n+1}_{k,i}=f^{n}_{k,i}-\frac{c_k\Delta t}{\Delta x}(g^{n}_{k,i+\half}-g^{n}_{k,i-\half}),\\
		&f^{s,n+1}_{k,i}=g^{n+1}_{k,i}=u^{n+1}_{i}\omega_k,
	\end{aligned}
\end{equation}
and
\begin{equation}\label{FV:convex-u}
	\begin{aligned}
		u^{f,n+1}_{i}=u^n_{i}
		-\frac{\Delta t}{\Delta x}\sum_{k}(c_kf^{n}_{k,i+\half}-c_kf^{n}_{k,i-\half}),\\
		u^{g,n+1}_{i}=u^n_{i}
		-\frac{\Delta t}{\Delta x}\sum_{k}(c_kg^{n}_{k,i+\half}-c_kg^{n}_{k,i-\half}).		
	\end{aligned}
\end{equation}
In the above equations, the update of $f^{f,n+1}_{k,i}$, $f^{g,n+1}_{k,i}$ and $f^{s,n+1}_{k,i}$ correspond to the kinetic physics processes of the particle free transport at cell interfaces, the particle collisions around cell interfaces, and particle collisions inside each cell, respectively. Similarly, the update of $u^{f,n+1}_{i}$ and $u^{g,n+1}_{i}$ correspond to the macroscopic physics processes of the particle free transport and particle collisions at cell interfaces. Due to the fact that $0\leq W(\dtdtau)\leq 1$, all the coefficients in the expression \eqref{FV:convex} are belong to the interval $[0,1]$. Thus, it is seen from \eqref{FV:convex} that the first-order UGKS is indeed a convex combination of sub-methods related to various physics processes.

Adopt the notation
\begin{equation}\label{notation-U-alpha}
	\bfU^{\alpha,n+1}
	=\big(\frac{\bff^{\alpha,n+1}_{-K}}{\sqrt{\omega_{-K}}},\dots,
	\frac{\bff^{\alpha,n+1}_{K}}{\sqrt{\omega_{K}}}\big)^\top,\quad
	\quad \alpha=f,g,s.
\end{equation}
where $\bff^{\alpha,n+1}_{k}=(f^{\alpha,n+1}_{k,1},\dots,f^{\alpha,n+1}_{k,I})$.

The expression \eqref{FV:convex} indicates the strong-stability preserving property of the first-order UGKS .

\begin{lem}\label{lem:convex}
	If the method given by \eqref{FV:convex}--\eqref{FV:convex-u} satisfies 
	\begin{equation}\label{condition-convex}
		\|\bfU^{\alpha,n+1}\|_{L^2}\leq \|\bfU^{n}\|_{L^2},\quad
		\quad \alpha=f,g,s,
	\end{equation}
\end{lem}
then it preserves the strong stability property
\begin{equation}\notag
	\|\bfU^{n+1}\|_{L^2}\leq \|\bfU^{n}\|_{L^2}.
\end{equation}
\begin{proof}
	Note that the norm $\|\cdot\|_{L^2}$ given by \eqref{def:norm} is a convex functional. Therefore,
	\begin{equation}\notag
		\begin{aligned}
			\|\bfU^{n+1}\|_{L^2}
			\leq &\frac{1}{1+\dtdtau}(1-W(\dtdtau))\|\bfU^{g,n+1}\|_{L^2}
			+W(\dtdtau)\|\bfU^{f,n+1}\|_{L^2}\\
			&+\frac{\dtdtau}{1+\dtdtau}\|\bfU^{s,n+1}\|_{L^2}\\
			\leq& \|\bfU^{n}\|_{L^2}.
		\end{aligned}
	\end{equation}
	The proof is completed.
\end{proof}

As a result, in order to derive the strong stability \eqref{ineq-main-f}, it suffices to show \eqref{condition-convex}, which is the topic of following subsections. 

\subsection{Particle free transport at cell interfaces}
This subsection is devoted to showing $\|\bfU^{f,n+1}\|_{L^2}\leq \|\bfU^{n}\|_{L^2}$, which is related to the physics process of the particle free transport at cell interfaces. Without loss of generality, we study the first-order finite volume scheme
\begin{equation}\label{FV:flux-f}
	f^{f,n+1}_{k,i}=f^{f,n}_{k,i}
	-\frac{c_k\Delta t}{\Delta x}(f^{n}_{k,i+\half}-f^{n}_{k,i-\half}),
\end{equation}
where
\begin{equation}\label{FV:flux-f1}
	f^{n}_{k,i+\half}=\frac{1}{2}(f^{f,n}_{k,i}+f^{f,n}_{k,i+1})
	-\frac{1}{2}\sign(c_k)(f^{f,n}_{k,i+1}-f^{f,n}_{k,i}).
\end{equation}
It suffices to show $\|\bfU^{f,n+1}\|_{L^2}\leq \|\bfU^{f,n}\|_{L^2}$, which implies $\|\bfU^{f,n+1}\|_{L^2}\leq \|\bfU^{n}\|_{L^2}$ by replacing the data $f^{f,n}_{k,i}$ with $f^{n}_{k,i}$ at the $n$-th time step.

We have the following lemma. 
\begin{lem}\label{lem:f}
	Suppose that the conditions \eqref{assumption-num-integration}, \eqref{assumption-ini-compatibility} and \eqref{boundary-periodic} hold. Then, if the CFL condition \eqref{CFL} holds, then the numerical solution of the method \eqref{FV:flux-f} satisfies
	\begin{equation}\label{L2norm-Uf}
		\|\bfU^{f,n+1}\|_{L^2}\leq \|\bfU^{f,n}\|_{L^2},
	\end{equation}
	where $\bfU^{f,n+1}$ is given in \eqref{notation-U-alpha}.
\end{lem}

\begin{proof}
	For each $k$, the scheme \eqref{FV:flux-f}--\eqref{FV:flux-f1} is regarded as a first-order upwind type finite volume scheme for the linear scalar advection equation
	\begin{equation}\notag
		\pt_t f_{k}+c_k\pt_x f_{k}=0.
	\end{equation}
	Under periodic boundary conditions \eqref{boundary-periodic}, for each $k$, it can be proved using the standard von-Neumann stability analysis method \cite{MM2005} that
	\begin{equation}\label{von-Neumann-analysis-1}
		\sum_{i=1}^{I} \Delta x|f^{f,n+1}_{k,i}|^2\leq \sum_{i=1}^{I}\Delta x |f^{f,n}_{k,i}|^2,
	\end{equation}
	provided that $\frac{c_k\Delta t}{\Delta x}\leq 1$.	The details for the proof of the estimate \eqref{von-Neumann-analysis-1} are presented in \Cref{sec:ap}. The estimate \eqref{von-Neumann-analysis-1} indicates 
	\begin{equation}\notag
		\big\|\frac{\bff^{f,n+1}_{k}}{\sqrt{\omega_k}}\big\|_{l^2}\leq 	\big\|\frac{\bff^{f,n}_{k}}{\sqrt{\omega_k}}\big\|_{l^2}, 
	\end{equation}
	which leads to \eqref{L2norm-Uf}.
	
	The proof is completed.
\end{proof}

\subsection{Particle collisions around cell interfaces}
\label{subsect-interface-g}
In this subsection, we proceed to show $\|\bfU^{g,n+1}\|_{L^2}\leq \|\bfU^{n}\|_{L^2}$, which is related to the physics process of particle collisions around cell interfaces. Without loss of generality, we study the first-order finite volume scheme
\begin{equation}\label{FV:flux-g}
	f^{g,n+1}_{k,i}=f^{g,n}_{k,i}-\frac{c_k\Delta t}{\Delta x}(g^{n}_{k,i+\half}-g^{n}_{k,i-\half}),
\end{equation}
where
\begin{equation}\label{FV:flux-g1}
	\begin{aligned}			&g^n_{k,i+\half}=u^{g,n}_{i+\half}\omega_k,\quad
		u^{g,n}_{i+\half}=\sum_{k}\Delta c \frac{c_k}{\sqrt{a^2+\frac{1}{2}\theta}} f^{g,n}_{k,i+\half},\\
		&f^{g,n}_{k,i+\half} = 
		\frac{1}{2}(f^{g,n}_{k,i}+f^{g,n}_{k,i+1})
		-\erf(\frac{a}{\sqrt{\theta}})(f^{g,n}_{k,i+1}-f^{g,n}_{k,i}).
	\end{aligned}
\end{equation}
It suffices to show $\|\bfU^{g,n+1}\|_{L^2}\leq \|\bfU^{g,n}\|_{L^2}$, which implies $\|\bfU^{g,n+1}\|_{L^2}\leq \|\bfU^{n}\|_{L^2}$ by replacing the data $f^{g,n}_{k,i}$ with $f^{n}_{k,i}$ at the $n$-th time step.

Rewrite the scheme \eqref{FV:flux-g}--\eqref{FV:flux-g1} as 
\begin{equation}\label{scheme-g}
	\frac{1}{\Delta t}\bigg(\frac{f^{g,n+1}_{k,i}}{\sqrt{\omega_k}}
	-\frac{f^{g,n}_{k,i}}{\sqrt{\omega_k}}\bigg)
	=\frac{1}{\sqrt{a^2+\frac{1}{2}\theta}}
	\sum_{l}\Delta c c_kc_l \sqrt{\omega_k\omega_{l}}
	\frac{1}{\Delta x}\bigg(\frac{f^{g,n}_{l,i+\half}}{\sqrt{\omega_l}}
	-\frac{f^{g,n}_{l,i-\half}}{\sqrt{\omega_l}}\bigg).
\end{equation}
The scheme given by \eqref{scheme-g} can be regarded as a finite volume discretization to the linear hyperbolic system
\begin{equation}\notag
	\pt_t \bfU = \bfA \pt_x \bfU, 
\end{equation}
where $\bfU=\big(\frac{f_{-K}}{\sqrt{\omega_{-K}}},\dots,
\frac{f_{K}}{\sqrt{\omega_{K}}}\big)^\top$, $\bfA$ is a $(2K+1)\times (2K+1)$ constant matrix given by
\begin{equation}\notag
	A_{kl}=\frac{1}{\sqrt{a^2+\frac{1}{2}\theta}}\Delta c c_kc_l \sqrt{\omega_k\omega_{l}}.
\end{equation} 

A key observation lies in that the matrix $\bfA$ can be expressed as
\begin{equation}\label{def:A}
	\bfA = \frac{1}{\sqrt{a^2+\frac{1}{2}\theta}}\bfb \bfb^\top,
\end{equation}
where $\bfb$ is a vector belong to $\R^{2K+1}$ and is defined as
\begin{equation}
	\bfb=(c_{-K}\sqrt{\omega_{-K}\Delta c},
	\dots,c_{K}\sqrt{\omega_{K}\Delta c})^\top.
\end{equation}
Thus, $\mathrm{Rank}\bfA=1$ and $\bfA$ has only one non-zero eigenvalue, i.e., 
\begin{equation}\notag
	\lambda_1
	= \frac{1}{\sqrt{a^2+\frac{1}{2}\theta}}\bfb^\top \bfb
	=\frac{1}{\sqrt{a^2+\frac{1}{2}\theta}}\sum_{k}\Delta c{c_k}^2\omega_k.
\end{equation}
In what follows we assume that the condition \eqref{assumption-num-integration} holds. Then, a unit left eigenvectors of $\bfA$ is
\begin{equation}\label{def:l_1}
	\bfl_1=\frac{1}{\sqrt{a^2+\frac{1}{2}\theta}}\bfb,
\end{equation}
which satisfies $\bfl_1^\top\bfA=\lambda_1\bfl_1^\top$ and 
\begin{equation}\notag
	\bfl_1\cdot\bfl_1
	= \frac{1}{a^2+\frac{1}{2}\theta}\bfb \cdot \bfb 
	=\frac{1}{a^2+\frac{1}{2}\theta}
	\sum_{k}\Delta c{c_k}^2\omega_k
	=1.
\end{equation}
The other left eigenvectors of $\bfA$ are correspond to the zero eigenvalue, and they consist of the vectors in space $\bfl_1^\perp$, i.e., the complementary space of $\bfl_1$. Let $\{\bfl_2, \bfl_3, \dots, \bfl_{2K+1}\}$ be a unit orthogonal basis of the space $\bfl_1^{\perp}$, and let $\bfL$ be the matrix of left eigenvectors, i.e.,
\begin{equation}\label{def:L}
	\bfL=\begin{pmatrix}
		\bfl_1^\top\\
		\vdots\\
		\bfl_{2K+1}^\top
	\end{pmatrix}.
\end{equation}	
Then, it is observed that $\bfl_{m}^\top\bfA=0$ for $m=2,\dots,2K+1$, and
\begin{equation}\notag
	\bfl_{k}\cdot \bfl_{l}=\begin{cases}
		1,\quad\quad \mathrm{if}\, k=l,\\
		0,\quad\quad \mathrm{else}.
	\end{cases}
\end{equation} 
Note the fact that if a square matrix has unit orthogonal rows, then it also has unit orthogonal columns. Therefore, $\bfL$ is a real unit orthogonal matrix satisfying
\begin{equation}\label{L-2norm-0}
	\|\bfL\|_{2}=\|\bfL^{-1}\|_{2}=1,
\end{equation}
where $\bfL^{-1}$ is the inverse matrix of $\bfL$ and $\|\cdot\|_{2}$ denotes the spectral norm of a matrix. Further, it is obtained by the Cauchy-Schwarz inequality that 
\begin{equation}\label{L-2norm-1}
	\|\bfL\|_{2}=\sup\{\bfx^\top\bfL\bfy: x, y\in\R^{2K+1}
	\; \mathrm{with}	\; \|\bfx\|_{2}=\|\bfy\|_{2}=1\},
\end{equation}
where $\|\cdot\|_2$ also denotes the Euclidean norm for vectors in Euclidean spaces. 

Let $\bffR=(R_1,\dots,R_{2K+1})$, where $R_{m}=\bfl_{m}^{\top}\bfU$ is the Riemann invariant corresponding to the left eigenvector $\bfl_{m}$, $m=1,\dots,2K+1$. The facts \eqref{L-2norm-0} and \eqref{L-2norm-1} imply that
\begin{equation}\notag
	\|\bffR\|_{2}=\|\bfU\|_{2}.
\end{equation} 

Based on the above observations, we are in a position to show the following lemma. 
\begin{lem}\label{lem:g}
	Suppose that the conditions \eqref{assumption-num-integration}, \eqref{assumption-ini-compatibility} and \eqref{boundary-periodic} hold. Then, if the CFL condition \eqref{CFL} holds, then the numerical solution of the method \eqref{FV:flux-g} satisfies
	\begin{equation}\label{L2norm-Ug}
		\|\bfU^{g,n+1}\|_{L^2}\leq \|\bfU^{g,n}\|_{L^2}.
	\end{equation}
	where $\bfU^{g,n+1}$ is given by \eqref{notation-U-alpha}.
\end{lem}
\begin{proof}
	The scheme \eqref{FV:flux-g} is rewritten as
	\begin{equation}\label{scheme-g-re}
		\frac{1}{\Delta t}\big(\bfU^{g,n+1}_{i}-\bfU^{g,n}_{i}\big)
		=\bfA\frac{1}{\Delta x}
		\big(\bfU^{g,n}_{i+\half}-\bfU^{g,n}_{i-\half}\big),
	\end{equation}
	where the matrix $\bfA$ is given in \eqref{def:A}, and
	\begin{equation}\notag
		\bfU^{g,n+1}_{i}
		=\bigg(\frac{f^{g,n+1}_{-K,i}}{\sqrt{\omega_{-K}}},\dots,
		\frac{f^{g,n+1}_{K,i}}{\sqrt{\omega_{K}}}\bigg)^\top,\quad
		\bfU^{g,n+1}_{i+\half}
		=\bigg(\frac{f^{g,n+1}_{-K,i+\half}}{\sqrt{\omega_{-K}}},\dots,
		\frac{f^{g,n+1}_{K,i+\half}}{\sqrt{\omega_{K}}}\bigg)^\top.
	\end{equation}
	
	Let the left eigenvector matrix $\bfL$ be the same as in \eqref{def:L}, and denote the corresponding Riemann invariants by 
	\begin{equation}\label{def:Riemann-invariant}
		R^{n}_{m,i}=\bfl_{m}^{\top}\bfU^{g,n}_{i}, \quad\quad m=1,\dots,2K+1,\quad i=1,\dots,I.
	\end{equation}
	To obtain \eqref{L2norm-Ug}, in what follows we first prove that
	\begin{equation}\label{estimate-R}
		\sum_{i=1}^{I}\sum_{m=1}^{2K+1}|R^{g,n+1}_{m,i}|^2	
		\leq \sum_{i=1}^{I}\sum_{m=1}^{2K+1}|R^{g,n}_{m,i}|^2.	
	\end{equation}	
	
	For $m=2,\dots, 2K+1$, due to the fact that $\bfl_m^\top\bfA=0$, multiply the equations \eqref{scheme-g-re} by $\bfl_m^\top$ from the left, leading to 
	\begin{equation}\label{estimate-R2}
		R^{n+1}_{m,i}=R^{n}_{m,i},\quad\quad m=2,\dots, 2K+1.
	\end{equation}
	
	For $m=1$, it follows from \eqref{def:l_1} that
	\begin{equation}\notag
		\sqrt{\Delta c}R^{n}_{1,i}=\sqrt{\Delta c}\bfl_1^\top\bfU^{g,n}_{i}=\sum_{k}\Delta c c_kf^{g,n}_{k,i}.
	\end{equation}
	Based on this, multiply the equations \eqref{scheme-g-re} by $\sqrt{\Delta c}\bfl_1^\top$ from the left, yielding
	\begin{equation}\label{scheme-GKS}
		F^{g,n+1}_{i}=F^{g,n}_{i}
		-\sqrt{a^2+\frac{1}{2}\theta}\frac{\Delta t}{\Delta x}
		\bigg(F^{g,n}_{i+\half}-F^{g,n}_{i-\half}\bigg),
	\end{equation}
	where
	\begin{equation}\label{flux:GKS}
		\begin{aligned}
			&F^{g,n+1}_{i}=\sum_{k}\Delta c c_kf_{k,i}^{g,n+1},\quad
			F^{g,n}_{i}=\sum_{k}\Delta c c_kf_{k,i}^{g,n},\\
			&F^{g,n}_{i+\half}=\frac{1}{2}(F^{g,n}_{i}+F^{g,n}_{i+1})
			-\erf(\frac{a}{\sqrt{\theta}})(F^{g,n}_{i+1}-F^{g,n}_{i}).
		\end{aligned}
	\end{equation}
	The method given by \eqref{scheme-GKS} is indeed a kinetic upwind finite volume scheme (c.f.\cite{TX2006}) applied to a scalar inviscid advection equation
	\begin{equation}\notag
		\pt_t F + \sqrt{a^2+\frac{1}{2}\theta} \pt_x F=0.
	\end{equation} 
	Thus, with periodic boundary conditions and based on the fact that $0\leq\erf(\frac{a}{\sqrt{\theta}})\leq1$, it can be proved using the standard von-Neumann stability analysis that (see \cite{TX2006} for the details of von-Neumann stability analysis on the kinetic upwind schemes for linear advection-diffusion equations) that
	\begin{equation}\label{estimate-R1-0}
		\sum_{i=1}^{I}\Delta x|F^{g,n+1}_{i}|^2\leq \sum_{i=1}^{I}\Delta x|F^{g,n}_{i}|^2,
	\end{equation} 
	provided that $	\sqrt{a^2+\frac{1}{2}\theta}\frac{\Delta t}{\Delta x}\leq \erf\big(\frac{a}{\sqrt{\theta}}\big)$. For the comprehensiveness of the paper, the details for the proof of the estimate \eqref{estimate-R1-0} are presented in \Cref{sec:ap}.
	
	Consequently, \eqref{estimate-R1-0} implies
	\begin{equation}\label{estimate-R1}
		\sum_{i=1}^{I}|R^{g,n+1}_{1,i}|^2\leq \sum_{i=1}^{I}|R^{g,n}_{1,i}|^2.
	\end{equation}
	
	Combining \eqref{estimate-R2} and \eqref{estimate-R1} gives \eqref{estimate-R}.
	
	Next, we are in a position to show 
	\begin{equation}\label{RUleft}
		\sum_{i=1}^{I}\|\bfU^{g,n+1}_i\|^2_{2} \leq  
		\sum_{i=1}^{I}\|\bffR^{n+1}_{i}\|^2_{2}.
	\end{equation} 
	where $\bffR^{n}_{i}=(R^{n}_{1,i},\dots,R^{n}_{2K+1,i})^\top$. 
	
	By the definition \eqref{def:Riemann-invariant}, it holds that $\bfU^{g,n+1}_{i}=\bfL^{-1}\bffR^{n+1}_{i}$, indicating
	\begin{equation}\notag
		\|\bfU^{g,n+1}_{i}\|_{2}\leq \|\bfL^{-1}\|_{2}\|\bffR^{n+1}_{i}\|_{2}, \quad\quad i=1,\dots,I.
	\end{equation}
	Therefore, the above inequality, together with the fact that $\|\bfL^{-1}\|_{2}=1$ (see \eqref{L-2norm-0}), gives \eqref{RUleft}.
	
	Similarly, noting that $\bffR^{n}_{i}=\bfL\bfU^{g,n}_{i}$ and $\|\bfL\|_{2}=1$, we have
	\begin{equation}\label{RUright}
		\sum_{i=1}^{I}\|\bffR^{n}_{i}\|^2_{2}\leq	\sum_{i=1}^{I}\|\bfU^{g,n}_i\|^2_{2}.
	\end{equation} 
	
	Finally, combining \eqref{RUleft}, \eqref{estimate-R} and \eqref{RUright} leads to
	\begin{equation}\notag
		\sum_{i=1}^{I}\|\bfU^{g,n+1}_i\|^2_{2} \leq  
		\sum_{i=1}^{I}\|\bffR^{n+1}_{i}\|^2_{2} \leq 
		\sum_{i=1}^{I}\|\bffR^{n}_{i}\|^2_{2}\leq	\sum_{i=1}^{I}\|\bfU^{g,n}_i\|^2_{2},
	\end{equation}
	which implies \eqref{L2norm-Ug}.
	
	The proof is completed.
\end{proof}

\subsection{Particle collisions inside each cell}	
In this subsection, we aim to show the estimate $\|\bfU^{s,n+1}\|_{L^2}\leq \|\bfU^{n}\|_{L^2}$, which is related to the physics process of particle collisions inside each cell. The corresponding sub-method reads
\begin{equation}\label{FV:source}
	\begin{aligned}
		&f^{s,n+1}_{k,i}=u^{n+1}_{i}\omega_k,\quad
		u^{n+1}_{i}=W(\dtdtau)u^{f,n+1}_{i}
		+(1-W(\dtdtau))u^{g,n+1}_{i}, \\
		&u^{g,n+1}_{i}=u^n_{i}
		-\frac{\Delta t}{\Delta x}\sum_{k}(c_kg^{n}_{k,i+\half}-c_kg^{n}_{k,i-\half}),\\
		&u^{f,n+1}_{i}=u^n_{i}
		-\frac{\Delta t}{\Delta x}\sum_{k}(c_kf^{n}_{k,i+\half}-c_kf^{n}_{k,i-\half}).
	\end{aligned}
\end{equation}

We have the following lemma. 
\begin{lem}\label{lem:s}
	Suppose that the conditions \eqref{assumption-num-integration}, \eqref{assumption-ini-compatibility} and \eqref{boundary-periodic} hold. Then, if the CFL condition \eqref{CFL} holds, then the numerical solution of the method \eqref{FV:source} satisfies
	\begin{equation}\label{L2norm-Us}
		\|\bfU^{s,n+1}\|_{L^2}\leq \|\bfU^{n}\|_{L^2}.
	\end{equation}
	where $\bfU^{s,n+1}$ is given by \eqref{notation-U-alpha}.
\end{lem}	

\begin{proof}
	Based on condition that the numerical integration is accurate in the sense of \eqref{assumption-num-integration}, the constraint-preserving property \eqref{constraint-preserving} and the equations \eqref{FV:convex}--\eqref{FV:convex-u} lead to
	\begin{equation}\label{scheme-s}
		\begin{aligned}
			&f^{s,n+1}_{k,i}=u^{n+1}_{i}\omega_k,\quad
			u^{n+1}_{i}=\sum_{k} \Delta c f^{pre,n+1}_{k,i},\\
			&f^{pre,n+1}_{k,i}=W(\dtdtau) f^{f,n+1}_{k,i}
			+(1-W(\dtdtau)) f^{g,n+1}_{k,i},
		\end{aligned}
	\end{equation}
	where $f^{pre,n+1}_{k,i}$ stands for the discrete distribution function predicted solely by considering the effect of numerical fluxes across cell interfaces.
	
	Adopt the notations
	\begin{equation}\notag
		\bfU^{pre,n+1}=\big(\frac{\bff^{pre,n+1}_{-K}}{\sqrt{\omega_{-K}}},\dots,
		\frac{\bff^{pre,n+1}_{K}}{\sqrt{\omega_{K}}}\big)^\top,
		\quad
		\bff^{pre,n+1}_{k}=(f^{pre,n+1}_{k,1},\dots,f^{pre,n+1}_{k,I}).
	\end{equation} 
	It follows from \Cref{lem:f} and \Cref{lem:g} that
	\begin{equation}\label{estimate:f-pre}
		\|\bfU^{pre,n+1}\|_{L^2}
		\leq W(\dtdtau)\|\bfU^{f,n+1}\|_{L^2}
		+(1-W(\dtdtau))\|\bfU^{g,n+1}\|_{L^2}
		\leq \|\bfU^{n}\|_{L^2}.
	\end{equation}
	
	It is obtained from \eqref{scheme-s} that
	\begin{equation}\notag
		\begin{aligned}
			\frac{f^{s,n+1}_{k,i}}{\sqrt{\omega_k}}
			=&\sum_{l}  \sqrt{\omega_k\Delta c}
			\sqrt{\omega_l\Delta c}
			\frac{f^{pre,n+1}_{l,i}}{\sqrt{\omega_l}}.
		\end{aligned}
	\end{equation}
	Based on the above equation and the condition $\sum_{k}\omega_k\Delta c=1$, the Cauchy-Schwartz inequality gives that
	\begin{equation}\notag
		\begin{aligned}
			\sum_{k}\big|\frac{f^{s,n+1}_{k,i}}{\sqrt{\omega_k}}\big|^2
			=&\big(\sum_{k}\omega_k\Delta c\big)
			\big|\sum_{l}\sqrt{\omega_l\Delta c}	\frac{f^{pre,n+1}_{l,i}}{\sqrt{\omega_l}}\big|^2\\
			=&\big|\sum_{l}\sqrt{\omega_l\Delta c}	\frac{f^{pre,n+1}_{l,i}}{\sqrt{\omega_l}}\big|^2\\
			\leq& \big(\sum_{l}\omega_l\Delta c\big)
			\sum_{l}\big|\frac{f^{pre,n+1}_{l,i}}{\sqrt{\omega_l}}\big|^2\\
			=&\sum_{l}\big|\frac{f^{pre,n+1}_{l,i}}{\sqrt{\omega_l}}\big|^2.
		\end{aligned}
	\end{equation}
	This implies
	\begin{equation}\label{estimate:f-s}
		\|\bfU^{s,n+1}\|_{L^2}\leq 	\|\bfU^{pre,n+1}\|_{L^2}.
	\end{equation}
	
	Combining \eqref{estimate:f-s} and \eqref{estimate:f-pre} leads to \eqref{L2norm-Us}.
	
	The proof is completed.
\end{proof}

\section{Conclusions}\label{sec:conclusion}	
Based on a direct modeling of physical laws in a control volume with limited cell resolution, the UGKS is a multiscale method for flow simulations in all flow regimes, in which the numerical flux is constructed using the integration solution of the kinetic model equation. This paper develop a first-order UGKS based on a linear kinetic model of BGK type. This model is able to reproduce the 1-D linear scalar advection-diffusion equation via the Chapman-Enskog expansion method. This paper rigorously proves the weighted $L^2$-stability of the first-order UGKS and derives a hyperbolic type CFL condition. The time step of the method neither suffers from restrictions of being less than the particle collision time, nor is it limited by parabolic type CFL conditions commonly used in the diffusion equations. 

For the proof, the creativity of this paper lies in that based on the ratio of the time step to the particle collision time, the update of distribution functions is viewed as a convex combination of sub-methods. Those sub-methods are in relation to various physics processes, such as the particle free transport and collisions, indicating the multiscale physics in the UGKS. The strong stability of the sub-methods are obtained by regarding them as finite volume discretizations to linear hyperbolic systems and using the associated Riemann invariants. Consequently, the strong stability preserving property gives the desired weighted $L^2$-stability of the first-order UGKS. 

Our approach provides a novel analytical framework to rigorously prove the strong stability of multiscale kinetic methods. The results serve as a first-step to analyze the stability of the UGKS for the full compressible gas dynamic.

\appendix
\section{Von Neumann stability analysis}\label{sec:ap}
In the appendix we present the details for applying the von Neumann stability analysis method to derive the estimates \eqref{von-Neumann-analysis-1} and \eqref{estimate-R1-0}.

We first consider the method \eqref{FV:flux-f}--\eqref{FV:flux-f1}. Without loss of generality, we present only the proof for the case with $c_k>0$. The results for other cases can be obtained in the same manner. For $c_k>0$, the method \eqref{FV:flux-f}--\eqref{FV:flux-f1} is rewritten as
\begin{equation}\label{ap:method-1}
	f^{f,n+1}_{k,i}=f^{f,n}_{k,i}
	-\frac{c_k\Delta t}{\Delta x}(f^{n}_{k,i}-f^{n}_{k,i-1}).
\end{equation}
In the following, we will investigate the stability properties in a $L^2$-setting following the stability analysis of von-Neumann for linear problems with periodic boundary conditions, see e.g. \cite{MM2005}. We investigate the effect of the method \eqref{ap:method-1} on harmonic waves with wave number $m$ and amplitude $\hat{f}^{f,n}_{k,m}$,
\begin{equation}
	f^{f,n}_{k,i}=\hat{f}^{f,n}_{k,m}e^{\im m (i\Delta x)},
\end{equation}
where $\im$ denotes the imaginary unit. Introducing the above Fourier expansion into the linear method \eqref{ap:method-1}, we obtain 
\begin{equation}
	\hat{f}^{f,n+1}_{k,m}=\lambda_k(m)\hat{f}^{f,n}_{k,m},
\end{equation}
where the amplification factor $\lambda_k(m)$ is given by
\begin{equation}
	\lambda_k(m)=1-\eta_k(1-e^{-\im m \Delta x}),\quad \eta_k=\frac{c_k\Delta t}{\Delta x}.
\end{equation}
It follows from a direct calculation that
\begin{equation}
	|\lambda_k(m)|^2=1-4\eta_k(1-\eta_k)\sin^2\big(\half m \Delta x\big)\leq 1.
\end{equation}
provided that $|\eta_k| \leq 1$. The above inequality implies the estimate \eqref{von-Neumann-analysis-1}.

Next, following the proof in \cite{TX2006}, we are in a position to show the estimate \eqref{estimate-R1-0}. Similarly, applying the Fourier expansion into the linear method \eqref{scheme-GKS}--\eqref{flux:GKS}, we obtain that
\begin{equation}
	\hat{F}^{g,n+1}_{m}=G(m)\hat{F}^{g,n}_{m},
\end{equation}
where $G(m)$ is the amplification factor. We skip the details of the calculation of the amplification factor and present the result in the form
\begin{equation}
	\begin{aligned}
		g(m):=&|G(m)|^2-1\\
		=&4\sin^2\frac{\xi}{2}
		\bigg(\beta(\beta-\erf\alpha)
		-\beta^2(1-(\erf\alpha)^2)\sin^2\frac{\xi}{2}\bigg),
	\end{aligned}
\end{equation}
where
\begin{equation}
	\xi=  m\Delta x,\quad	\beta=\sqrt{a^2+\frac{\theta}{2}}\frac{\Delta t}{\Delta x},\quad
	\alpha=\frac{a}{\sqrt{\theta}}.
\end{equation}
It suffices to prove 
\begin{equation}
	\max_{\xi\in[0,\pi]}g(\xi)\leq 0,
\end{equation}
which leads to
\begin{equation}
	\beta(\beta-\erf\alpha)
	-\beta^2(1-(\erf\alpha)^2)\sin^2\frac{\xi}{2}\leq 0.
\end{equation}
Since the above expression is linear in $\sin^2\frac{\xi}{2}$, we 
evaluate it at two positions $\xi=0$ and $\xi=\pi$ to obtain two necessary and sufficient conditions
\begin{equation}\label{ap:conditions}
	\beta(\beta-\erf\alpha)\leq 0\quad\quad\text{and}\quad\quad 
	\beta\erf\alpha(-1+\beta\erf\alpha)\leq 0.
\end{equation}
Based on the assumption $a>0$, we have $\beta>0$ and $\erf\alpha>0$. Thus, the conditions \eqref{ap:conditions} are reduced to 
\begin{equation}
	\beta \leq \frac{1}{\erf\alpha}\quad\quad\text{and}\quad\quad 	\beta \leq \erf\alpha.
\end{equation}  
Note that $\erf\alpha\leq 1$, hence the second condition is more restrictive and gives the estimate \eqref{estimate-R1-0}.


\vspace{2mm}
\noindent {\bf Acknowledgments:}

This current work was supported by National Key R$\&$D Program of China (Grant Nos. 2022YFA1004500),
National Natural Science Foundation of China (12172316, 92371107),
and Hong Kong research grant council (16301222, 16208324). We would like to thank Wei Liu for helpful discussions.
 
\vspace{2mm}


\end{document}